\documentclass[11pt]{birkjour}
 \newtheorem{thm}{Theorem}[section]
 \newtheorem{cor}[thm]{Corollary}
 \newtheorem{lem}[thm]{Lemma}
 \newtheorem{prop}[thm]{Proposition}
 \theoremstyle{definition}
 \newtheorem{defn}[thm]{Definition}
 \newtheorem{ex}[thm]{Example}
 \theoremstyle{remark}
 \newtheorem{rem}[thm]{Remark}
 \numberwithin{equation}{section}
\usepackage{color}
\begin{document}

%
%
%
%
%
%
%
%
%

\title[A new look at optimal dual problem related to fusion frames]
 {A New Look at  Optimal Dual Problem Related to Fusion Frames}


\author[F. Arabyani Neyshaburi]{Fahimeh Arabyani-Neyshaburi}
\address{Department of Mathematics and Computer Sciences, Hakim Sabzevari University, Sabzevar, Iran.}
\email{fahimeh.arabyani@gmail.com}

\author[ A. Arefijamaal]{Ali Akbar Arefijamaal}
\address{Department of Mathematics and Computer Sciences, Hakim Sabzevari University, Sabzevar, Iran.}
\email{arefijamaal@hsu.ac.ir, arefijamaal@gmail.com}


\vspace{2cm}
\begin{abstract}
The purpose of this work is to examine the structure of optimal dual fusion frames and get more flexibility in the use of  dual fusion frames   for  erasures of subspaces.
We deal with optimal  dual fusion frames with  respect to  different definitions of duality and compare the advantages of these approaches.
In addition, we introduce a new concept so called  partial optimal dual which   involves  less time and computation for detecting  optimal dual  for erasures in known locations.
 Then we  study   the relationship between local and global optimal duals  by  partial optimal duals which leads to some applicable results. In the sense that, we obtain an overcomplete  frame and a family of associated  optimal duals by a given Riesz fusion basis. We present some examples to exhibit  the effect  of  error rate when dual fusion frames  are applied in  reconstruction.
\end{abstract}

\maketitle

\section {Introduction}
Finite frame theory  has been recently  a major tool in many outstanding applications in engineering and applied mathematics such as filter bank theory, packet encoding, signal and image processing due to their resilience to additive noise and erasures \cite{Ar13, Ben06,  Bod04,  Bod03,  Bod05,  Bol98, sto}. Indeed, the redundancy property of frames reduces the errors of reconstructed  signals when erasures fall out.
Suppose a frame $F$ is preselected for encoding in a communication system in a finite dimensional Hilbert space. Optimal dual problem asks for finding the dual frames of $F$ that minimize the maximal  reconstruction error when some coefficients in transmission  have been erased or reshaped. In \cite{lopez}, the authors presented this problem and gave some sufficient conditions which the canonical dual is the unique optimal dual. Then  the authors \cite{leng} obtained several results under which the canonical dual  is optimal or not optimal for any $r$-erasures. Moreover,  the characterization of  extreme points in the  set of  all $1$-loss optimal duals  and    some discussions on conditions that an alternate dual frame is either optimal or not optimal  dual  can be found in \cite{Aa2018}. Some generalizations of optimal dual  problem have been done on reconstruction systems which are related to $g$-frames and fusion frames. See \cite{Bod03, massey}.

Fusion frames are one of the most important generalization of frames which provides efficient frameworks  for a wide range of applications that cannot be modeled by discrete frames \cite{Cassaz, Cas08, Ole01, sensor, hear}. So
motivated by  several applications in distributed processing, parallel processing and overall building efficient algorithms for reconstruction by fusion frames, Heineken et. al. generalized  the optimal dual problem  for fusion frames using $Q$-component preserving  duals \cite{ Hein}. Some  sufficient conditions  which the canonical dual fusion frame is an optimal dual or not optimal as a $Q$-component preserving dual  was presented in \cite{morilas}.
 In this work, we are interested to set this idea in the context of dual fusion frames introduced by P. G$\breve{\textrm{a}}$vru\c{t}a  \cite{Gav07}.  One of our aims  is to compare the effect of this approach  with $Q$-component preserving optimal duals.  In addition, we would like to work more with alternate dual fusion frames.

The organization of this article is as follows. In Section 2, the basic definitions and notations of  frame theory, fusion frame theory, optimal duals  and some of their fundamental results will be given.
In Section 3, we study the  differences between  optimal duals under two types of definitions of duality in this setting. This also shows our motivation for working with  dual fusion frames defined by  P. G$\breve{\textrm{a}}$vru\c{t}a. Then, in Section 4 we give some sufficient conditions for determining the existence and robustness of  optimal  dual fusion frames. Moreover, we discuss   the relation between local and global optimal duals, by a new notion called partial optimal duals.  We present some examples to show  the effect  of the  error rate when dual fusion frames are used to  reconstruction, in Section 5.

\section {Preliminaries and notations}
Let $\mathcal{H}$ be an $n$-dimensional   Hilbert space and $I_{m}=\{1, 2, ... ,m\}$. A family $F:=\{f_{i}\}_{i\in I_{m}}\subseteq\mathcal{H}$ is called a
\textit{frame} for $\mathcal{H}$ whenever ${\textit{span}}\{f_{i}\}_{i\in I_{m}}=\mathcal{H}$ so obviously $m\geq n$.
Three important operators associated with finite family $F$ are defined as follows. The \textit{synthesis operator} $\theta_{F}:
 l^{2}(I_{m})\rightarrow \mathcal{H}$ is defined by $\theta_{F}\lbrace c_{i}\rbrace= \sum_{i\in I_{m}}
c_{i}f_{i}$, the \textit{analysis operator} that is  the adjoint
of $\theta_{F}$,  given by $\theta_{F}^{*}f= \lbrace \langle f,f_{i}\rangle\rbrace_{i
\in I_{m}}$, mapping $\mathcal{H}$ into $l^{2}(I_{m})$. Also, the frame operator, given by  $S_{F}=
\theta_{F}\theta_{F}^{*}$  or equivalently $S_{F}f=\sum_{i\in I_{m}} \langle f,f_{i}\rangle f_{i}$. An easily argument shows that $F$ is a frame if and only if the frame operator is invertible.
The fact that $\mathcal{H}$ is finite dimensional   Hilbert space implies that the analysis and synthesis operators are continuous and specially the continuous function $f\mapsto \Vert \theta_{F}^{*}f\Vert$ is non-zero on compact unit sphere in $\mathcal{H}$.
Thus, $F=\{f_{i}\}_{i\in I_{m}}$ is a frame for $\mathcal{H}$ if and only if
 there exist  positive constants
 $A$, $ B$ such that
\begin{eqnarray}\label{Def frame}
A\|f\|^{2}\leq \Vert \theta_{F}^{*}f\Vert=\sum_{i\in I_{m}}|\langle f,f_{i}\rangle|^{2}\leq
B\|f\|^{2},\qquad (f\in \mathcal{H}).
\end{eqnarray}
The constants $A$ and $B$ are called the \textit{frame bounds}.  If $A=B$, the frame $F$ is
called  a \textit{tight frame}, and in the case of $A=B=1$
 it is a  \textit{Parseval frame}.
Applying a frame  $F$ any elements of $\mathcal{H}$ can be recovered from the frame coefficients $\{\langle f,f_{i}\rangle\}_{i\in I_{m}}$
\begin{eqnarray*}
f=\sum_{i\in I_{m}} \langle f,f_{i}\rangle S_{F}^{-1}f_{i}=\sum_{i\in I_{m}} \langle f,S_{F}^{-1}f_{i}\rangle f_{i}=\sum_{i\in I_{m}} \langle f,S_{F}^{-1/2}f_{i}\rangle S_{F}^{-1/2}f_{i}.
\end{eqnarray*}
Hence, $\{S_{F}^{-1/2}f_{i}\}_{i\in I_{m}}$ is a Parseval frame and $\{S_{F}^{-1}f_{i}\}_{i\in I_{m}}$ is a frame that is called the canonical dual.
A family $G:=\{g_i\}_{i\in I_{m}}\subseteq\mathcal{H}$ is called to be a
\textit{ dual}
 for $\{f_{i}\}_{i\in I_{m}}$ if $\theta_{G}\theta_{F}^{*}=I_{\mathcal{H}}$. It is well known that $\{g_i\}_{i\in I_{m}}$ is a dual  frame of $\{f_i\}_{i\in I_{m}}$ if and only if $g_{i} = S_{F}^{-1}f_{i}+u_{i}$, for all $i\in I_{m}$ where $\{u_i\}_{i\in I_{m}}$ satisfies
\begin{eqnarray}\label{dual relation}
\sum_{i\in I_{m}} \langle f, u_{i}\rangle f_{i} = 0, \quad (f\in \mathcal{H}).
\end{eqnarray}
Every   frame $F$ with linearly independent vectors   is called a Riesz basis and in case $F$ is not a Riesz basis it  is called overcomplete.
We refer the reader to  \cite{Chr08} for more information on frame theory.
 The optimal dual problem, one of the most important problems in   frame theory,  brings up the following problem:
let $F = \{f_{i}\}_{i\in I_{m}}$ be a frame for $n$-dimensional Hilbert space
$\mathcal{H}$, find a dual frame of $F$ that minimize  the reconstruction errors when erasures occur. If $G =
\{g_{i}\}_{i\in I_{m}}$ is a dual of $F$ and $\Lambda \subset I_{m}$, then
the error operator $E_{\Lambda}$ is defined by

\begin{eqnarray*}
E_{\Lambda} = \sum_{i\in \Lambda}g_{i} \otimes f_{i} = \theta_{G}D\theta^{*}_{F},
\end{eqnarray*}
where $D$ is a $m\times m$ diagonal
matrix with $d_{ii}=1$ for $i\in \Lambda$ and $0$ otherwise. Let
\begin{eqnarray}\label{01.optimal}
 d_{r}(F,G) = \max \{ \Vert \theta_{G}D\theta^{*}_{F}\Vert : D\in \mathcal{D}_{r}\} = \max\{\|E_{\Lambda}\| : \vert \Lambda \vert = r\}, \end{eqnarray}
in which $\vert \Lambda\vert$ is the cardinality of $\Lambda$, the norm used in (\ref{01.optimal}) is the operator norm, $1\leq r<m$ is a natural number and $\mathcal{D}_{r}$ is the set of all $m \times m$ diagonal matrices
with $r$ $1 ' s$ and $n-r$ $0 ' s$. Then $ d_{r}(F,G) $ is the largest possible error
when $r$-erasures occur. Indeed, $G$ is called an optimal dual frame of $F$
for $1$-erasure or $1$-loss optimal dual  if \begin{eqnarray}\label{1-erasure def}
 d_{1}(F,G) = \min \left\{  d_{1}(F,Y) : Y \textit{ is a dual of  }         F\right\}. \end{eqnarray}
Inductively, for $r>1$, a dual frame $G$ is called an \textit{optimal dual} of $F$  for
$r$-erasures ($r$-loss optimal dual) if it is optimal for $(r-1)$-erasures and
\begin{eqnarray*}
 d_{r}(F,G) = \min \left\{  d_{r}(F,Y) : Y \textit{ is a dual of  }         F\right\}. \end{eqnarray*}
Also, some studies on optimal dual problem have been done in which  the error operator was considered by   different measurements instead of the operator norm in (\ref{01.optimal}). See \cite{AAS, Saliha}. This comes from the fact that  dependent on  applications using of different measurements for the error rate can simplify computations and be more suitable.
In what follows, we prefer the Frobenius norm for fusion frame setting due to comparing our results with \cite{Hein}.

  Fusion frame theory is a fundamental mathematical theory   introduced in  \cite{Cas04} to model sensor networks perfectly.
Although, recent studies show
 that fusion frames provide effective frameworks  not only for modeling of sensor networks but also    for a variety of applications that cannot be modeled by discrete frames. In the following, we review basic definitions  of fusion frames.

 Let $\{W_i\}_{i\in I_{m}}$ be a family of closed subspaces of   $\mathcal{H}$ and $\{\omega_i\}_{i\in I_{m}}$ a family of
weights, i.e. $\omega_i>0$, $i\in I_{m}$. Then  $\{(W_i,\omega_i)\}_{i\in
I_{m}}$ is  called a \textit{fusion frame} for $\mathcal{H}$ if ${\textit{span}}\{W_{i}\}_{i\in I_{m}}=\mathcal{H}$ equivalently, there exist
constants $0<A\leq B<\infty$ such that
\begin{eqnarray}\label{Def. fusion}
A\|f\|^{2}\leq \sum_{i\in I_{m}}\omega_i^2\|\pi_{W_i}f\|^2\leq
B\|f\|^{2},\qquad (f\in \mathcal{H}),
\end{eqnarray}
where $\pi_{W_{i}}$ denotes the orthogonal projection from Hilbert space $\mathcal{H}$ onto a closed subspace $W_{i}$. The constants $A$ and $B$ are called the \textit{fusion frame
bounds}.
Also,  a fusion frame is called $A$-\textit{tight}, if $A=B$,   \textit{Parseval} if $A= B= 1$ in $(\ref{Def. fusion})$, \textit{$\omega$-uniform}  if  $\omega_{i}=\omega$ for all $i\in I_{m}$  and we abbreviate $1$- uniform fusion frames as $\{W_i\}_{i\in I_{m}}$.
A family of closed subspaces $\{W_i\}_{i\in I_{m}}$ is called an orthonormal basis for $\mathcal{H}$ when $\oplus_{i\in I_{m}} W_{i} = \mathcal{H}$. Furthermore,
the sequence  $\{(W_i,\omega_i)\}_{i\in
I_{m}}$ is called a \textit{Riesz fusion  basis} whenever it is a complete family  in  $\mathcal{H}$ and  there exist positive constants $A$, $B$ so that for every finite subset $J\subset I_{m}$ and arbitrary vector $f_{i}\in W_i$, we have
\begin{eqnarray*}\label{Def fusion}
A\sum_{i\in J}\| f_i\|^{2} \leq \left\| \sum_{i\in J} f_i\right\|^{2} \leq B\sum_{i\in J}\| f_i\|^{2}.
\end{eqnarray*}
Let $\{(W_i,\omega_i)\}_{i\in
I_{m}}$ be a  sequence of subspaces and consider the Hilbert space
\begin{eqnarray*}
\mathcal{W}:=\sum_{i\in I_{m}}\oplus W_{i} =\left\{\{f_i\}_{i\in I}:f_i\in W_i \right\}.
\end{eqnarray*}
The \textit{synthesis operator} $T_{W}\in B(\mathcal{W} ,\mathcal{H})$ given  by
\begin{equation*}
T_{W}\left(\{f_i\}_{i\in I_{m}}\right)=\sum_{i\in I_{m}}\omega_if_i,\qquad
\left(\{f_{i}\}_{i\in I_{m}}\in \mathcal{W}\right).
\end{equation*}
Its adjoint  $T_{W}^{*}\in B(\mathcal{H}, \mathcal{W})$, which is called the \textit{analysis
operator}, is obtained by
\begin{eqnarray*}
T_{W}^{*}(f)=\left\{\omega_{i}\pi_{W_{i}}(f)\right\}_{i\in I_{m}},\qquad (f\in
\mathcal{H}).\end{eqnarray*}
and the \textit{fusion frame operator}
$S_{W}\in B(\mathcal{H})$ is defined by $S_{W
}f=\sum_{i\in I_{m}}\omega_i^{2}\pi_{W_i}f$, which is a bounded,
invertible and positive operator \cite{Cas04}.

Throughout this paper, we
 use  $(W,w)$ to denote a fusion frame
$\lbrace (W_{i}, \omega_{i})\rbrace_{i\in I_{m}}$ in a finite dimensional Hilbert space $\mathcal{H}$, if there exists $i\in I_{m}$ so that $W_{i}\neq \mathcal{H}$ we call $(W,w)$ a non-trivial fusion frame and in this paper, we consider non-trivial case for all fusion frames.
Also, if $(W,w)$ is a fusion
frame and $F_{i}=\{f_{i,j}\}_{j\in J_{i}}$ is a frame for $W_{i}$ for
$i=1,2,... ,m$. Then $(W,w,F)$ is called a fusion frame system, where
$F=\{F_{i}\}_{i\in I_{m}}$. Let $\mathcal{H},\mathcal{K}$ be two   Hilbert spaces, we use of $B(\mathcal{H},\mathcal{K})$ for denoting of all bounded linear operators of $\mathcal{H}$ into $\mathcal{K}$ and
we abbreviate $B(\mathcal{H},\mathcal{H})$ by $B(\mathcal{H})$. Finally, we use of $\mathcal{L}_{T_{W}^{*}}$,  for the set of all left inverses of $T_{W}^{*}$,  $\Vert . \Vert_{F}$, for the Frobenius  norm  and $\Vert . \Vert$, for the operator  norm.

 Let $(W,w)$ be a fusion frame for Hilbert space $\mathcal{H}$ then for reconstruction of the elements of  $\mathcal{H}$
there are some approaches towards definition of
 dual fusion frames, one approach  was presented in \cite{ Hei14,Hein}.
\begin{defn}\label{Q dual}
Let  $(W,w)$ and   $(V,v)$ be two fusion frames for $\mathcal{H}$. Then $(V,v)$
 is called a  $Q$-dual of $(W,w)$ if there exists a linear operator $Q\in B(\mathcal{W},\mathcal{V})$ such that
\begin{eqnarray}
T_{V}QT^{*}_{W} = I_{\mathcal{H}}.
\end{eqnarray}
\end{defn}
Consider the operator $M_{J,W}:\mathcal{W}\rightarrow \mathcal{W}$, as $M_{J,W}\{f_{j}\}_{j\in I_{m}}=\{\chi_{J}(j)f_{j}\}_{j\in I_{m}}$, that $\chi_{J}$ is the characteristic function  and so clearly  $M_{J,W}$  is a self-adjoint operator. We simply write $M_{J}$ if it is clear to which $W$ we refer to. Also, we abbreviate $M_{\{j\},W}=M_{j,W}$ and $M_{\{j\}}=M_{j}$.
\begin{defn}
Let  $Q\in L(\mathcal{W},\mathcal{V})$
\begin{itemize}\item[(1)] If $QM_{j,W}\mathcal{W}\subseteq M_{j,V}\mathcal{V}$ for all $j\in I_{m}$ then $Q$ is called block-diagonal.
\item[(2)] If $QM_{j,W}\mathcal{W}= M_{j,V}\mathcal{V}$ for all $j\in I_{m}$ then $Q$ is called component preserving.
\end{itemize}
\end{defn}
In Definition \ref{Q dual} we say that $(V,v)$ is a $Q$-block-diagonal (component preserving) dual fusion frame of $(W,w)$  if $Q$ is block-diagonal (component preserving). To characterize component preserving dual fusion frames, the  authors in \cite{Hein} considered some notations as follows. Let $\mathcal{A}\in B(\mathcal{W}, \mathcal{H})$, and $v$ be a collection of weights, $V_{i}=\mathcal{A}M_{i}\mathcal{W}$, for all $i\in I_{m}$ and consider the linear transformation
\begin{eqnarray}
Q_{\mathcal{A},v}:\mathcal{W}\rightarrow \mathcal{V}, \quad Q_{\mathcal{A},v}\{f_{j}\}_{j\in I_{m}}=\left\{\dfrac{1}{\nu_{i}} \mathcal{A}M_{i}\{f_{j}\}_{j\in I_{m}}\right\}_{i\in I_{m}}.
\end{eqnarray}
Then by these notations for a given fusion frame we have a complete characterization of $Q$-component preserving dual fusion frames.
\begin{thm}\cite{Hein}\label{charac Qdual}
Let  $(W,w)$ be a fusion frame for $\mathcal{H}$. Then  $(V,v)$ is a $Q$-component preserving dual fusion frame of  $(W,w)$ if and only if $V_{i}=\mathcal{A}M_{i}\mathcal{W}$ for all $i\in I_{m}$ and $Q=Q_{\mathcal{A},v}$ for some $\mathcal{A}\in \mathcal{ L}_{T^{*}_{W}}$. Moreover, any element of $ \mathcal{ L}_{T^{*}_{W}}$ is of the form $T_{V}Q$ where $(V,v)$ is some $Q$-component preserving dual fusion frame of $(W,w)$.
\end{thm}

Now let $(W,w)$ be a fusion frame for $\mathcal{H}$ and  $(V,v)$ a $Q$-dual fusion frame of  $(W,w)$. Then every
$f\in \mathcal{H}$ can be reconstructed by $T_{V}QT_{W}^{*}f=f$. Suppose
that $J\subseteq  I_{m}$ and the data vectors corresponding to the
subspace $\{W_{i}\}_{i\in J}$ are erased. Then the reconstruction gives
$T_{V}QM_{ I_{m}\setminus J}T_{W}^{*}f$.  In \cite{Hein}, the authors presented some approaches for finding
those dual fusion frames of $(W,w)$ which are optimal for these situations. In the following we state their method in summary.

Consider a fix $r\in I_{m}$ and take $P_{r}^{m}:= \{J \subseteq
I_{m} : |J| = r\}$, also notice that $M_{J}=I_{\mathcal{W}}\setminus
M_{ I_{m}-J}$, where $\mathcal{W}=\sum_{i\in I_{m}}\oplus W_{i}$. In this
case, the worst case error is
\begin{eqnarray}\label{worst case}
 \|e(r,W,V)\|_{\infty} = \max_{J\in
P_{r}^{m}}\|T_{V}QM_{J}T_{W}^{*}\|_{F}. \end{eqnarray}
 The set of all $1$-loss
optimal dual fusion frames under this measure, $D_{1}^{\infty}(W,w)$, is defined as the set
of all  $Q$-duals  $(V,v)$ of $(W,w)$ such that
\begin{eqnarray*}
 e_{1}^{\infty}(W,w):=\|e(1,W,V)\|_{\infty}=\inf\left\{\|e(1,W,Z)\|_{\infty} : (Z,z) \quad\textit{is a $Q$-dual  of} \quad (W,w)\right\}
 \end{eqnarray*}
 and inductively, the set of all
$r$-loss optimal $Q$-dual fusion frames is defined as
\begin{eqnarray*} D_{r}^{\infty}(W,w)= \{(V,v)\in D_{r-1}^{\infty
}(W,w) :   \|e(r,W,V)\|_{\infty}
= e_{r}^{\infty}(W,w)\}.
\end{eqnarray*}
In  \cite{Hein}, the authors considered this concept for $Q$-component preserving dual fusion frames and showed that if $(W,w)$ is a fusion frame
 for $\mathcal{H}$, then
  \begin{eqnarray*} \mathfrak{A}= \lbrace \mathcal{A}\in \mathcal{ L}_{T^{*}_{W}}:  ,
\max_{i\in I_{m}}\Vert \mathcal{A}M_{i}T^{*}_{W}\Vert_{F}
 = min_{B\in \mathcal{ L}_{T^{*}_{W}}} \max_{i\in I_{m}}\Vert
BM_{i}T^{*}_{W}\Vert_{F} \rbrace,
\end{eqnarray*}
is a non empty, compact and convex set. Then by this fact and using Theorem \ref{charac Qdual}  they could present some results under which the canonical dual is the unique optimal as a $Q$-component preserving dual fusion frame. These results in \cite{morilas}  was extended for the case that  the canonical dual,  as a $Q$-component preserving dual, is optimal or not optimal  for probabilistic erasures.

The other  approach to define  dual fusion frames was defined by P.
G$\breve{\textrm{a}}$vru\c{t}a in \cite{Gav07}. In the sense that,
a  sequence $(V,v)$ of subspaces is
called a \textit{dual} fusion frame of $(W,w)$ if
\begin{eqnarray}\label{Def:alt operatory}
T_{V}\phi_{vw}T_{W}^{*}=\sum_{i\in I_{m}}\omega_{i}\nu_{i}\pi_{V_i}S_{W}^{-1}\pi_{W_i}=I_{\mathcal{H}},
\end{eqnarray}
where $\phi_{vw}:\mathcal{W}\rightarrow\mathcal{V}$ is defined as
\begin{eqnarray*}
\phi_{vw}\{f_{i}\}_{i\in I}=\{\pi_{V_{i}}S_{W}^{-1}f_{i}\}_{i\in I_{m}}, \quad \left(\{f_{i}\}_{i\in I}\in \mathcal{W}\right).
\end{eqnarray*}
The family $\{(S_{W}^{-1}W_i,\omega_i)\}_{i\in I_{m}}$  is called the \textit{canonical dual} of
$(W,w)$.

In this work, we are going to continue   optimal dual  problem specially for alternate dual fusion frames in  this concept.
\begin{rem} \label{2.9..}
It is worth to note that, these dual fusion frames  can be considered as a $Q$-block diagonal dual. Indeed, using $(\ref{Def:alt operatory})$ we observe that
 \begin{eqnarray*}
\phi_{vw}M_{j}\mathcal{W}\subseteq M_{j}\mathcal{V},
\end{eqnarray*}
for all $j\in I_{m}$ so $\phi_{vw}$  is block diagonal.
However,  Example \ref{ex01} shows that  $\phi_{vw}$  is not necessarily component preserving.
\end{rem}


\section{Opposite relations among optimal dual and $P$-optimal dual fusion frames}

In this section, we survey similarities and specially differences between  optimal duals and $P$-optimal duals. Throughout the paper we will work with the worst case of error as in $(\ref{worst case})$, so we define our consideration of optimality  as follows;
 \begin{defn} \label{1.4}
Let  $(W,w)$  be a fusion frame
 for $\mathcal{H}$ and $(V,v)\in D_{W}$. We say that  $(V,v)$ is $1$-loss optimal dual of $(W,w)$   whenever
\begin{eqnarray*}
\max_{i\in I_{m}} \Vert  \omega_{i}\nu_{i}\pi_{V_{i}}S_{W}^{-1}\pi_{W_{i}}\Vert_{F}=\inf\left\{\max_{i\in I_{m}} \Vert  \omega_{i}z_{i}\pi_{Z_{i}}S_{W}^{-1}\pi_{W_{i}}\Vert_{F} : (Z,z)\in D_{W}\right\}
\end{eqnarray*}
\end{defn}
Inductively,  we can extend the above definition for any $r$-erasures. More precisely, $(V,v)$ is called an optimal dual of $(W,w)$ for any $r$-erasures, whenever it is  an optimal dual of $(W,w)$ for any $(r-1)$-erasures and
\begin{eqnarray}\label{optim}
\max_{J\in
P_{r}^{m}}\|T_{V}\phi_{vw}M_{J}T_{W}^{*}\|_{F}=\inf\left\{\max_{J\in
P_{r}^{m}}\|T_{Z}\phi_{zw}M_{J}T_{W}^{*}\|_{F}: (Z,z)\in D_{W}\right\}. \end{eqnarray}
As we mentioned in Remark \ref{2.9..},
 $ \phi_{vw}$  is block diagonal but  not necessarily component preserving.  In this respect, if a dual is optimal in the set of  all  $Q$-component preserving dual fusion frames we call it $P$-optimal dual and it is called an optimal dual if it is
 optimal in the notion $(\ref{optim})$.
To simplify the notations,  we use of $D_{W}$ and $OD^{r}_{W}$ $(OD_{W})$ to denote the set of all  duals of $W$ and optimal duals of $W$ under $r$-erasures ($1$-erasure ) under dual definition as in  $(\ref{Def:alt operatory})$, respectively. The following lemmas are useful for later use.
\begin{lem}\label{lem3}\cite{Gav07}
Let   $\mathcal{H}$ be a Hilbert spaces and $U\in B(\mathcal{H})$. Also, let $V$ be a closed subspace of $\mathcal{H}$. Then
$\pi_{V}U^{*}=\pi_{V}U^{*}\pi_{\overline{UV}}$.
Moreover, the following are equivalent
\begin{itemize}
\item[(i)]$U\pi_{V} =\pi_{\overline{UV}}U.$
\item[(ii)]$U^{*}U V\subseteq V.$
\end{itemize}
\end{lem}
In the following lemma, we suspect that the set $\mathfrak{D}$ is compact. Indeed it is obviously bounded but maybe it is not closed in general. However, the compactness is a secondary problem, we only want to show that $OD_{W}$ is  non-empty.
\begin{lem}\label{existence.R}
Let $(W,w)$ be a  fusion
frame for finite dimensional Hilbert space $\mathcal{H}$. Then $OD^{r}_{W}$ is non-empty under any $r$-erasures.
\end{lem}
\begin{proof}
We first  show that  $OD_{W}$ is a non-empty set. To this  end, it is sufficient to prove that
\begin{eqnarray*}
\mathcal{O}=\left\{T_{Y}\phi_{yw}:  (Y,y)\in D_{W},   \max_{i\in I_{m}}\Vert T_{Y}\phi_{yw}M_{i}T^{*}_{W} \Vert_{F}=\inf_{ (Z,z)\in D_{W}} \max_{i\in I_{m}}\Vert T_{Z}\phi_{zw}M_{i}T^{*}_{W} \Vert_{F}\right\}
\end{eqnarray*}
is a non-empty set. The mapping $\Vert .\Vert_{w_{r}}: B(\mathcal{W},\mathcal{H})\rightarrow \mathbb{R}^{+}$  defined as $\Vert A\Vert_{w_{r}}= \max_{J\in P_{r}^{m}}\Vert AM_{J}T^{*}_{W} \Vert_{F}$ is a norm on $B(\mathcal{W},\mathcal{H})$ by Theorem 3.1 of \cite{morilas}.  On the other hand,
\begin{eqnarray*}
\mathfrak{D}=\{T_{Y}\phi_{yw}:  (Y,y)\in D_{W},   \Vert T_{Y}\phi_{yw}\Vert_{w_{1}}\leq \Vert S_{W}^{-1}T_{W} \Vert_{w_{1}}\}
\end{eqnarray*}
is a non-empty and  compact subset of $B(\mathcal{W},\mathcal{H})$. Hence,  $\Vert .\Vert_{w_{1}}$ attains its infimum  on $\mathfrak{D}$, i.e., there exists a dual fusion frame  $(V,v)$ of $(W,w)$ so that
\begin{eqnarray*}
\Vert T_{V}\phi_{vw}\Vert_{w_{1}}=\inf_{ T_{Y}\phi_{yw}\in \mathfrak{D}}\Vert T_{Y}\phi_{yw}\Vert_{w_{1}}.
\end{eqnarray*}
Since $OD_{W} \subseteq \{(Y,y) : T_{Y}\phi_{yw}\in \mathfrak{D}\}$,  therefore $(V,v)\in OD_{W}$, i.e., $\mathcal{O}\neq \emptyset$. The rest of the proof  is done by a simple induction. Let the set of all optimal dual fusion frames for any $(r-1)$-erasures,$OD_{W}^{r-1}$, is non-empty. Then a similar argument as in the above shows that the  set
\begin{eqnarray*}
\left\{T_{Y}\phi_{yw}:  (Y,y)\in OD_{W}^{r-1},  \Vert T_{Y}\phi_{yw}\Vert_{w_{r}}=\inf_{ (Z,z)\in D_{W}} \Vert T_{Z}\phi_{zw}\Vert_{w_{r}}\right\}
\end{eqnarray*}
is non-empty and consequently $OD^{r}_{W}$ is non-empty.
\end{proof}

We can associate to  every $Q$-block diagonal dual fusion frame $(V,v)$ of $(W,w)$ a $Q$-preserving dual fusion frame. See Remark 3.6 of \cite{Hein}. In the following, we  state this result for dual fusion frames as mutual relation.
\begin{lem}\label{lem0}
Let $(W,w)$ and $(V,v)$  be two fusion
frames for $\mathcal{H}$. Then $(V,v)$ is a dual fusion frame of $(W,w)$ if and only if $X:=\{(\pi_{V_{i}}S_{W}^{-1}W_{i},\nu_{i})\}_{i\in I_{m}}$ is a $\phi_{vw}$- component preserving dual of $(W,w)$.
\end{lem}
\begin{proof}
Suppose $(V,v)$ is a dual fusion frame of  $(W,w)$. So   $\mathcal{A}=T_{V}\phi_{vw}$ is a left inverse of $T_{W}^{*}$, $Q_{\mathcal{A},v}=\phi_{vw}$ and  $\mathcal{A}M_{i}\mathcal{W}=\pi_{V_{i}}S_{W}^{-1}W_{i}$, for all $i\in I_{m}$. Thus,
 by
Remark 3.6 of \cite{Hein} we imply that  $X$ is a $\phi_{vw}$-component preserving dual of $(W,w)$.  Conversely, if  $X$ is  a $\phi_{vw}$- component preserving dual of $(W,w)$ then $T_{X}\phi_{vw}\in \mathcal{ L}_{T^{*}_{W}}$ and
\begin{eqnarray*}
T_{X}\{\pi_{V_{i}}S_{W}^{-1}f_{i} \}_{I_{m}}=\sum_{i\in I_{m}}\nu_{i}\pi_{V_{i}}S_{W}^{-1}f_{i}=T_{V}\{\pi_{V_{i}}S_{W}^{-1}f_{i} \}_{I_{m}}, \quad (\{f_{i}\}_{i\in I_{m}}\in \mathcal{W}),
\end{eqnarray*}
Thus, $T_{V}\phi_{vw}T^{*}_{W}=T_{X}\phi_{vw}T^{*}_{W}=I_{\mathcal{H}}$ and this  completes the proof.
\end{proof}
Applying the above lemma, if the canonical dual is a $P$-optimal dual then it is an optimal dual fusion frame. Indeed,  for every dual fusion frame  $(V,v) $ of $(W,w)$
\begin{eqnarray}\label{3.2n}
\max_{i\in I_{m}}\Vert T_{V}\phi_{vw}M_{i}T^{*}_{W} \Vert_{F}=
\max_{i\in I_{m}}\Vert T_{X}\phi_{vw}M_{i}T^{*}_{W} \Vert_{F} \geq
\max_{i\in I_{m}}\Vert \omega_{i}^{2}S_{W}^{-1}\pi_{W_{i}}\Vert_{F},
\end{eqnarray}
where $X=\{(\pi_{V_{i}}S_{W}^{-1}W_{i},\nu_{i})\}_{i\in I_{m}}$.
Similarly, if  $(V,v) $ is a $p$-optimal dual, as a  $\phi_{vw}$-preserving dual, of $(W,w)$ then  $(V,v) \in OD_{W}$.
 However, there are several essential differences between  optimal dual and $P$- optimal dual fusion frames. These differences are due to the fact that a component preserving dual is not necessarily a dual fusion frame and vice versa. For a simple example we consider a special case of Example 6.3 in  \cite{Hein}. Let $\mathcal{H}=\mathbb{R}^{3}$, $W_{1}=(1,0,0)^{\perp}$, $W_{2}=(0,1,0)^{\perp}$ and $\omega_{1}=\omega_{2}=1$. Then $W:=\{(W_{i},\omega_{i})\}_{i=1}^{2}$ is a fusion frame for  $\mathcal{H}$. Put $V_{1}={\textit{span}}\{(0,1,0),(1,2,-1/2)\}$, $V_{2}={\textit{span}}\{(1,0,0),(-1,-2,3/2)\}$ and $\nu_{1}=\nu_{2}=1$. Then the mapping $\mathcal{A}:\sum_{i=1}^{2}\oplus W_{i}\rightarrow \mathcal{H}$ given by
\begin{eqnarray*}
\mathcal{A}((0,x_{2},x_{3}),(y_{1},0,y_{3}))=(x_{3}+y_{1}-y_{3}, x_{2}+2x_{3}-2y_{3}  ,-1/2x_{3}+3/2y_{3} )
\end{eqnarray*}
is a left inverse of $T_{W}^{*}$ and
$V:=\{(V_{i},\nu_{i})\}_{i=1}^{2}$ is a $Q_{\mathcal{A},v}$-preserving dual of $W$. However, a straightforward computation shows that $V$ is not a dual fusion frame of $W$. More precisely,
$S_{W}^{-1}(a,b,c)=(a,b,c/2)$, for all $(a,b,c)\in \mathbb{R}^{3}$ and so
\begin{eqnarray*}
\sum_{i=1}^{2}\pi_{V_{i}}S_{W}^{-1}\pi_{W_{i}}(a,b,c)=\left(a-\dfrac{c}{5}, b-\dfrac{6c}{25}, \dfrac{7c}{25}\right).
\end{eqnarray*}
Using these discussions one implies that if  a dual fusion frame is  optimal dual then  it is not necessarily a $P$-optimal dual and vice versa.
As mentioned above, only under a limited and  special condition if a dual fusion frame is $P$-optimal dual then it is also an optimal dual. However, even in this case,  the uniqueness cannot be  preserved  in general. See Example \ref{ex01}.  In the next remark we observe more cases of a fusion frame $W$ where $\vert OD_{W}\vert\geq 2$.

\begin{rem}\label{lemun}
Assume that $(W,w)$ is a  fusion
frame for $\mathcal{H}$.  Applying Lemma \ref{existence.R}
  $OD_{W}\neq \emptyset$ for any $r$-erasures.
 So, let $(V,v)\in OD^{r}_{W}$.  If  there exists $ i\in I_{m}$ so that
$0\neq (S_{W}^{-1}W_{i})^{\perp}\subseteq V_{i}$ then we confront two cases. If
$(S_{W}^{-1}W_{i})^{\perp}= V_{i}$ we take $Z_{i}=\{0\}$ and $Z_{j}=V_{j}$ for $j\neq i$. Then obviously  $(Z,v)\in OD^{r}_{W}$.  Also, in case $(S_{W}^{-1}W_{i})^{\perp}\subset V_{i}$ we  have that  $V_{i}=(S_{W}^{-1}W_{i})^{\perp} \oplus (S_{W}^{-1}W_{i}\cap V_{i})$. Consider $Z_{i}=S_{W}^{-1}W_{i}\cap V_{i}$ and $Z_{j}=V_{j}$ for $j\neq i$, then  $(Z,v)\in OD^{r}_{W}$.

Moreover, if
 $V_{i}^{\perp}\cap (S_{W}^{-1}W_{i})^{\perp} \neq \{0\}$ for some  $i\in I_{m}$. Then we can choose a non zero element $u\in V_{i}^{\perp}\cap (S_{W}^{-1}W_{i})^{\perp}$. Take $Z_{i}=V_{i}\oplus \textit{span}\{u\}$ and $Z_{j}=V_{j}$ for $j\neq i$. Then  $(Z,v)\in D_{W}$. Moreover,
\begin{eqnarray*}
\max_{i\in J}\Vert T_{V}\phi_{vw}M_{J}T^{*}_{W} \Vert_{F}=\max_{i\in J}\Vert T_{Z}\phi_{zw}M_{J}T^{*}_{W} \Vert_{F},
\end{eqnarray*}
for every $J\in P_{r}^{m}$. Thus $(Z,v)\in OD^{r}_{W}$ and is different from $(V,v)$.
\end{rem}
By the above explanations we   obtain the following results that gives sufficient conditions for finding $1$-loss  optimal dual fusion frames.
\begin{thm}\label{finalcor}
 Let $(W,w)$ be a fusion
frame for $\mathcal{H}$. Consider
\begin{eqnarray*}
c=\max\{\omega_{i}^{2}\Vert S_{W}^{-1}\pi_{W_{i}} \Vert_{F}, i\in I_{m}\}
\end{eqnarray*}
 $\Lambda_{1}=\{i\in I_{m} : \omega_{i}^{2}\Vert S_{W}^{-1}\pi_{W_{i}} \Vert_{F}=c \}$, $\Lambda_{2}=I_{m}\setminus \Lambda_{1}$ and $\mathcal{H}_{j}=\operatorname{span} \cup_{i\in \Lambda_{j}} W_{i}$, $j=1,2$. If $\mathcal{H}_{1}\cap \mathcal{H}_{2}=\{0\}$ and $ \lbrace (W_{i}, \omega_{i})\rbrace_{i\in \Lambda_{2}}$ is a Riesz fusion basis for $\mathcal{H}_{2}$, then $(S_{W}^{-1}W,w)\in OD_{W}$,  but not the unique optimal one.
\end{thm}
\begin{proof}
Using Theorem 3.3 of  \cite{morilas}  the canonical dual  is the unique $P$-optimal dual of  $(W,w)$ and so it is an optimal dual fusion frame, by the assertions after Lemma \ref{lem0}. Thus, one can easily deduce the result  by Remark \ref{lemun}.
\end{proof}
\begin{cor}
 Let $(W,w)$ be an $\alpha$-tight fusion
frame for $\mathcal{H}$ so that $ \omega_{i}^{2}\sqrt{dimW_{i}}=c$, for all $i\in I_{m}$.
Then $(V,v)\in OD_{W}$, for every $(V,v)\in D_{W}$ where $\max_{i\in I_{m}}\dfrac{\nu_{i}}{\omega_{i}}\leq 1$.
\end{cor}
\begin{proof}
We first note that the canonical dual  is  an optimal dual of  $(W,w)$ by Theorem \ref{finalcor}. Assume $(V,v)\in D_{W}$ so that $\max_{i\in I_{m}}\dfrac{\nu_{i}}{\omega_{i}}\leq 1$  then
\begin{eqnarray*}
\max_{i\in I_{m}} \Vert \omega_{i}\nu_{i} \pi_{V_{i}}S_{W}^{-1}\pi_{W_{i}} \Vert_{F} &=& 1/\alpha \max_{i\in I_{m}}\omega_{i}\nu_{i} \sqrt{tr(\pi_{W_{i}} \pi_{V_{i}}\pi_{W_{i}})}\\
&=&1/\alpha \max_{i\in I_{m}}\omega_{i}\nu_{i} \sqrt{tr(\pi_{V_{i}}\pi_{W_{i}})}\\
&\leq&1/\alpha \max_{i\in I_{m}}\omega_{i}\nu_{i} \sqrt{tr(\pi_{W_{i}})}\\
&=&1/\alpha \max_{i\in I_{m}}\dfrac{\nu_{i}}{\omega_{i}}  \omega_{i}^{2}\sqrt{dimW_{i}}\\
&\leq&c/\alpha \max_{i\in I_{m}}\dfrac{\nu_{i}}{\omega_{i}}\\
&\leq& c/\alpha = \max_{i\in I_{m}}  \omega_{i}^{2}\Vert S_{W}^{-1}\pi_{W_{i}} \Vert_{F}.
\end{eqnarray*}
Thus $(V,v)\in OD_{W}$, as required.
\end{proof}
\begin{rem}
If  $(W,w)$ is  a Riesz fusion basis
 for $\mathcal{H}$ then $(V,v)$  is a dual ($Q$-preserving dual) fusion frame of $(W,w)$ if and only if $S_{W}^{-1}W_{i}\subseteq V_{i}$, $i\in I_{m}$. See \cite{arefi1, Hein}.
Hence, every dual  fusion frame of $(W,w)$ is clearly an optimal dual. However, the canonical dual is the unique $P$-optimal dual of  $(W,w)$. Indeed, using the notations in Theorem \ref{finalcor}  and the fact that $(W,w)$ is  a Riesz fusion basis yield  $\{(W_{i}, \omega_{i})\}_{i\in\Lambda_{2}}$ is a Riesz fusion basis for $\mathcal{H}_{2}$ and $\mathcal{H}_{1} \cap \mathcal{H}_{2}=\{0\}$. Thus, the canonical dual is the unique $P$-optimal dual of  $(W,w)$ by   Theorem 3.3 of  \cite{morilas}. In addition, there exist some non-Riesz fusion frames with several optimal duals but the unique $P$-optimal dual, see   Example \ref{ex01}.
\end{rem}

\section{Optimal and partial optimal  dual fusion frames}
 According to the   differences between optimal  and $P$-optimal duals mentioned in Section 2, in the sequel, we survey more on  optimal dual fusion frames and their constructions. We present some sufficient condition, for building optimal dual fusion frames.
Specially, we introduce a new concept called partial optimal  dual fusion frame  and using that we study the relation among local and global optimal duals. Also, we get an overcomplete frame with a family of  optimal duals by a given Riesz fusion basis.
 Finally, we examine optimal dual fusion frames under operator perturbations.

Motivating  the notations in Theorem \ref{finalcor}, we define the following symbols to get sufficient conditions in order that  a dual fusion frame is optimal. So, let $(W,w)$ be a fusion frame of $\mathcal{H}$ and $(V,v)\in D_{W}$. Denote
\begin{eqnarray*}
c_{v}=\max\{\omega_{i}\nu_{i}\Vert \pi_{V_{i}} S_{W}^{-1}\pi_{W_{i}} \Vert_{F}, i\in I_{m}\},
\end{eqnarray*}
 $\Lambda_{1,v}=\{i\in I_{m} : \omega_{i}\nu_{i}\Vert \pi_{V_{i}} S_{W}^{-1}\pi_{W_{i}} \Vert_{F}=c_{v} \}$, $\Lambda_{2,v}=I_{m}\setminus \Lambda_{1,v}$ and $\mathcal{H}_{j,v}=\operatorname{span} \left\{\cup_{i\in \Lambda_{j,v}} W_{i}\right\}$, $j=1,2$.
\begin{prop}
Let $(W,w)$ be a fusion frame of $\mathcal{H}$ and $(V,v)\in D_{W}$ so that  $\{(W_{i}, \omega_{i})\}_{i\in\Lambda_{1,v}}$ is a Riesz fusion basis for $\mathcal{H}_{1,v}$ and $\mathcal{H}_{1,v} \cap \mathcal{H}_{2,v}=\{0\}$. Then $(V,v)$ is  a $1$-loss  optimal dual of $(W,w)$.
\end{prop}
\begin{proof}
Suppose that $(Z,z)\in D_{W}$, then
\begin{eqnarray*}
&&T_{W}M_{\Lambda_{1,v}}(T_{Z}\phi_{zw}-T_{V}\phi_{vw})^{*}+T_{W}M_{\Lambda_{2,v}}(T_{Z}\phi_{zw}-T_{V}\phi_{vw})^{*}\\
&=& T_{W}(T_{Z}\phi_{zw}-T_{V}\phi_{vw})^{*}\\
&=&T_{W}\phi_{zw}^{*}T^{*}_{Z}-T_{W}\phi_{vw}^{*}T^{*}_{V}=0.
\end{eqnarray*}
So, the hypothesis $\mathcal{H}_{1,v} \cap \mathcal{H}_{2,v}=\{0\}$ assures that   $T_{W}M_{\Lambda_{1,v}}(T_{Z}\phi_{zw}-T_{V}\phi_{vw})^{*}=0$, and consequently, $T_{V}\phi_{vw}M_{i}=T_{Z}\phi_{zw}M_{i}$, for all $i\in \Lambda_{1,v}$. Hence,
\begin{eqnarray*}
\max_{i\in I_{m}} \Vert T_{Z}\phi_{zw}M_{i}T^{*}_{W}\Vert_{F} &\geq& \max_{i\in \Lambda_{1,v}} \Vert T_{Z}\phi_{zw}M_{i}T^{*}_{W}\Vert_{F} \\
&=& \max_{i\in \Lambda_{1,v}} \Vert T_{V}\phi_{vw}M_{i}T^{*}_{W}\Vert_{F}\\
&=& \max_{i\in I_{m}} \Vert T_{V}\phi_{vw}M_{i}T^{*}_{W}\Vert_{F}.
\end{eqnarray*}
Thus, $(V,v)\in OD_{W}$ as required.
\end{proof}
A version of the above proposition is a sufficient condition under which the canonical dual is optimal dual for probabilistic erasures \cite{morilas}.    Probability model first of all was introduced in \cite{Han-leng}  by J. Leng et al. for finding optimal dual frames when probabilistic erasures occur. In fact, they considered different weights for the coefficients of $\theta_{F}^{*}f$ according to their degree of loss possibility.
 Then,  by some examples illustrated the advantages of using the  probability  optimal dual frames over the optimal dual frames when the coefficients with large erasure probability  are lost.

Herein, we consider  the other  problem:
given a finite frame $F$ so that the probability of elimination of $r$ coefficients  of $\theta_{F}^{*}f$ is near to $1$ (or certainly we know that which $r$ coefficients  have been lost).  In the other words, we suppose that the receiver knows which coefficients have been received. Then
 we would like to find optimal dual frames of $F$ just for elimination of these $r$ elements. This scheme avoids checking the existence of optimal dual frames for all erasures when partial erasures occur.

\begin{defn}
Suppose that $(W,w)$ is a fusion frame of $\mathcal{H}$. Then we say  $(V,v)$ is a partial optimal dual fusion frame of $(W,w)$ for   $r$-erasures if  $(V,v)$ is an  optimal dual fusion frame of $(W,w)$ for the elimination  of  some $r$-elements of $(W,w)$.  Equivalently,  $(V,v)$ is a partial optimal dual  of $(W,w)$ for   $r$-erasures if there exists $J\in P_{r}^{m}$ so that
\begin{eqnarray*}
 \Vert T_{V}\phi_{vw}M_{J}T_{W}^{*}\Vert_{F}=\inf\left\{\Vert T_{Z}\phi_{zw}M_{J}T_{W}^{*}\Vert_{F}: (Z,z) \in D_{W}\right\}.
\end{eqnarray*}
\end{defn}
Also,  the notion of  partial optimal duality can be considered for discrete frames  as a new concept. More precisely, let $F = \{f_{i}\}_{i\in I_{m}}$ be a frame for $n$-dimensional Hilbert space
$\mathcal{H}$ and $G\in D_{F}$ then $G$ is called a partial optimal dual   of $F$ for   $r$-erasures if  there exists a $k \times k$ diagonal matrix $D$
with $r$ $1 ' s$ and $n-r$ $0 ' s$ so that
\begin{eqnarray*}
\Vert \theta_{G}D\theta_{F}^{*}\Vert_{F} = \inf\left\{ \Vert \theta_{X}D\theta_{F}^{*}\Vert_{F} : X\in D_{F}\right\}.
\end{eqnarray*}
 In the next theorem, using this notion we present a connection between local and global optimal duals.
\begin{thm}\label{finalthmm}
Let  $(W,w)$  be a fusion frame for $n$-dimensional Hilbert space
$\mathcal{H}$ and $\{e_{j}\}_{j\in I_{n}}$ be an  orthonormal basis of $\mathcal{H}$.  Then the following assertions hold;
\begin{itemize}
\item[(1)]
If $(V,v)$ is a  fusion frame so that $G:=\{\nu_{i}\pi_{V_{i}}e_{j}\}_{j\in I_{n},i\in I_{m}}$ is a patrial optimal dual frame  for  $n$-erasures as  $\{\omega_{i}\pi_{W_{i}}S_{W}^{-1}e_{j}\}_{j\in I_{n}}$, $i\in I_{m}$ of  the frame $F:=\{\omega_{i}\pi_{W_{i}}S_{W}^{-1}e_{j}\}_{j\in I_{n},i\in I_{m}}$,  then $(V,v)$ is a $1$-loss optimal dual fusion frame of $(W,w)$
\item[(2)] If $(W,w)$ is a Riesz fusion basis then  all dual frames of $F$ can be considered as a partial optimal dual frame of $F$ for   $n$-erasures.
\end{itemize}
\end{thm}
\begin{proof}
Using Proposition 3.3 of \cite{Ar01} follows that $(V,v)$ is a dual fusion frame of $(W,w)$ if and only if $G=\{\nu_{i}\pi_{V_{i}}e_{j}\}_{j\in I_{n},i\in I_{m}}$ is a   dual of $F=\{\omega_{i}\pi_{W_{i}}S_{W}^{-1}e_{j}\}_{j\in I_{n},i\in I_{m}}$. Assume $G$ is a   patrial optimal dual frame  for  $n$-erasures as  $\{\omega_{i}\pi_{W_{i}}S_{W}^{-1}e_{j}\}_{j\in I_{n}}$, $i\in I_{m}$ of $F$. Then for every  $(Z,z)\in D_{W}$ the sequence $\mathcal{Z}=\{z_{i}\pi_{Z_{i}}e_{j}\}_{j\in I_{n},i\in I_{m}}$ is a   dual  frame of $F$. Consider $\Lambda_{i}=\{(i,j) : j\in I_{n}\}$  and the operator $D_{\Lambda_{i}}:\mathbb{C}^{nm}\rightarrow \mathbb{C}^{nm}$ defined
by $D_{\Lambda_{i}}\{c_{k,j}\}_{k\in I_{m},j\in I_{n}}=\{\chi_{\Lambda_{i}}(k,j)c_{k,j}\}_{k\in I_{m},j\in I_{n}}$, for all $i\in I_{m}$. Then we have
\begin{eqnarray*}
 \max_{i\in I_{m}}\Vert T_{V}\phi_{vw}M_{i}T_{W}^{*}\Vert_{F}&=& \max_{i\in I_{m}}\left\Vert \theta_{G}D_{\Lambda_{i}}\theta^{*}_{F}\right\Vert_{F}\\
&\leq& \max_{i\in I_{m}}\left\Vert\theta_{\mathcal{Z}}D_{\Lambda_{i}}\theta^{*}_{F}\right\Vert_{F}\\
&=& \max_{i\in I_{m}}\Vert T_{Z}\phi_{vw}M_{i}T_{W}^{*}\Vert_{F}.
\end{eqnarray*}
Hence,  $(V,v)$ is a $1$-loss optimal dual fusion frame of $(W,w)$. For proving $(2)$, we first obtain the canonical dual of $F$.
To this end, let  $S_{F}$ denotes the frame operator of $F$. Then
\begin{eqnarray*}
S_{F}= \sum_{i\in I_{m}} \omega_{i}^{2}\pi_{W_{i}}S_{W}^{-2}\pi_{W_{i}},
\end{eqnarray*}
and consequently
\begin{eqnarray*}
S_{F}\omega_{i}\pi_{S_{W}^{-1}W_{i}}e_{j}&=&\sum_{k\in I_{m}}\omega_{i}\omega_{k}^{2}\pi_{W_{k}}S_{W}^{-2}\pi_{W_{k}}\pi_{S_{W}^{-1}W_{i}}e_{j}\\
&=& \omega_{i}^{3}\pi_{W_{i}}S_{W}^{-2}\pi_{W_{i}}\pi_{S_{W}^{-1}W_{i}}e_{j}\\
&=& \omega_{i}\pi_{W_{i}}S_{W}^{-2}\sum_{k\in I_{m}} \omega_{k}^{2}\pi_{W_{k}}\pi_{S_{W}^{-1}W_{i}}e_{j}\\
&=& \omega_{i}\pi_{W_{i}}S_{W}^{-2}S_{W}\pi_{S_{W}^{-1}W_{i}}e_{j}\\
&=& \omega_{i}\pi_{W_{i}}S_{W}^{-1}\pi_{S_{W}^{-1}W_{i}}e_{j}\\
&=& \omega_{i}\pi_{W_{i}}S_{W}^{-1}e_{j},
\end{eqnarray*}
for every $i\in I_{m}$, $j\in I_{n}$,  where the second equality is due to the assumption that $W$ is a Riesz fusion basis and is orthogonal with its canonical dual, see \cite{Cas04, shamsabadi}, and the last equality is obtained by Lemma \ref{lem3}. Therefore the canonical dual of $F$ is obtained as follows
\begin{eqnarray*}
S_{F}^{-1}F=\{S_{F}^{-1} \omega_{i}\pi_{W_{i}}S_{W}^{-1}e_{j}\}_{j\in I_{n}, i\in I_{m}}=\{ \omega_{i}\pi_{S_{W}^{-1}W_{i}}e_{j}\}_{j\in I_{n}, i\in I_{m}}.
\end{eqnarray*}
So, for every dual frame $G$ of $F$ we can write $G=\{\omega_{i}\pi_{S_{W}^{-1}W_{i}}e_{j}+u_{i,j}\}_{j\in I_{n},i\in I_{m}}$ where
\begin{eqnarray*}
\sum_{j\in I_{n},i\in I_{m}}\langle  . , u_{i,j}\rangle \omega_{i}\pi_{W_{i}}S_{W}^{-1}e_{j}=0.
\end{eqnarray*}
The fact that $(W,w)$ is a Riesz fusion basis,  implies that
\begin{eqnarray}\label{232}
\pi_{W_{i}}\sum_{j\in I_{n}}\langle . , u_{i,j}\rangle S_{W}^{-1}e_{j}=0, \quad (i\in I_{m}).
\end{eqnarray}
Therefore, we can write
\begin{eqnarray*}
\Vert \theta_{G}D_{\Lambda_{i}}\theta^{*}_{F}\Vert_{F} &=&\left\Vert \sum_{j\in I_{n}} \langle . ,  \omega_{i}\pi_{W_{i}}S_{W}^{-1}e_{j} \rangle (\omega_{i}\pi_{S_{W}^{-1}W_{i}}e_{j}+u_{i,j}) \right\Vert_{F}\\
&=&\left\Vert \sum_{j\in I_{n}} \langle . ,  \omega_{i}\pi_{W_{i}}S_{W}^{-1}e_{j} \rangle \omega_{i}\pi_{S_{W}^{-1}W_{i}}e_{j} \right\Vert_{F}\\
&=& \Vert \theta_{S_{F}^{-1}F} D_{\Lambda_{i}}\theta^{*}_{F}\Vert _{F},
\end{eqnarray*}
for all $i\in I_{m}$. The above computations show that all dual frames of $F$ have the same error rate and so are optimal for  any $n$-erasures as $\{\omega_{i}\pi_{W_{i}}S_{W}^{-1}e_{j}\}_{j\in I_{n}}$, $i\in I_{m}$. This follows the desired result.
\end{proof}
An immediate result of the above theorem is as follows;\begin{cor}
 Suppose $(W,w)$  is a fusion frame of  Hilbert space
$\mathcal{H}$ and $\{e_{j}\}_{j\in I_{n}}$ is an  orthonormal basis of $\mathcal{H}$.
If $(V,v)$ is a  fusion frame so that $G:=\{\nu_{i}\pi_{V_{i}}e_{j}\}_{j\in I_{n},i\in I_{m}}$ is an optimal dual frame of $F:=\{\omega_{i}\pi_{W_{i}}S_{W}^{-1}e_{j}\}_{j=1,i\in I_{m}}^{n}$,   for any $n$-erasures  then $(V,v)$ is a $1$-loss optimal dual fusion frame of $(W,w)$.
\end{cor}
In the following  result we get a family of   optimal dual frame pairs with operator norm $(\ref{01.optimal})$  by a given  Riesz fusion basis.
\begin{thm}\label{orthonormal}
 Let $(W,1)$ be a Riesz fusion basis  of $\mathcal{H}$. Then $F=\{\pi_{S_{W}^{-1/2}W_{i}}e_{j}\}_{j\in I_{n},i\in I_{m}}$ is a Parseval frame with a family of optimal dual frames  as $ \{\pi_{V_{i}}e_{j}\}_{j\in I_{n},i\in I_{m}}$, for some orthonormal basis  $\{e_{j}\}_{j\in I_{n}}$ of $\mathcal{H}$  and all   sequences $(V,1)$ of subspaces which satisfy  $S_{W}^{-1/2}W_{i}\subseteq V_{i}$, $i\in I_{m}$.
\end{thm}
\begin{proof}
Since $(W,1)$ is a Riesz fusion basis, the family $\{(S_{W}^{-1/2}W_{i},1)\}_{i\in I_{m}}$ is an orthonormal fusion basis and so $F$ is a Parseval frame for $\mathcal{H}$.
Consider an element $\alpha_{j}\in S_{W}^{-1/2}W_{j}$ for all $j\in I_{m}$ and put $e_{j}=\alpha_{j}/\Vert \alpha_{j}\Vert$. Then $\{e_{j}\}_{j\in I_{m}}$ is an orthonormal subset of $\mathcal{H}$ and so it can be extended to an  orthonormal basis  $\{e_{j}\}_{j\in I_{n}}$ for $\mathcal{H}$ .
Thus, $\max_{i\in I_{m}, j\in I_{n}} \Vert \pi_{S_{W}^{-1/2}W_{i}}e_{j}\Vert^{2}=1$.  Hence,
\begin{eqnarray*}
\Lambda_{1}&=&\{(i, j): i\in I_{m}, j\in I_{n},   \Vert \pi_{S_{W}^{-1/2}W_{i}}e_{j}\Vert^{2}=1\}\\
&=&\{(i, j): i\in I_{m}, j\in I_{n},  e_{j}\in S_{W}^{-1/2}W_{i}\}.
\end{eqnarray*}
The fact that $(S_{W}^{-1/2}W,1)$ is an orthonormal fusion basis along with considering
$\Lambda_{2}=I_{m}\setminus \Lambda_{1}$, $H_{k}=\textit{span}\{\pi_{S_{W}^{-1/2}W_{i}}e_{j}\}_{(i,j)\in  \Lambda_{k}}$, $k=1,2$    assure that $H_{1}\cap H_{2}=\{0\}$ and $\{\pi_{S_{W}^{-1/2}W_{i}}e_{j}\}_{(i,j)\in  \Lambda_{1}}$ is linearly independent. Hence, the canonical dual of $F$ a $1$-loss optimal dual, \cite{leng}. Moreover, for every  family $(V,1)$  that $S_{W}^{-1/2}W_{i}\subseteq V_{i}$,  $i\in I_{m}$ we obtain $\max_{i\in I_{m}, j\in I_{n}}\Vert \pi_{V_{i}}e_{j}\Vert\Vert\pi_{S_{W}^{-1/2}W_{i}}e_{j}\Vert=1$, i.e., $G$ is also a $1$-loss optimal dual of $F$.
\end{proof}
The above theorem  immediately implies that if $(W,1)$ is an orthonormal fusion basis and the
sequence $(V,1)$ satisfies  $W_{i}\subseteq V_{i}$,  $i\in I_{m}$.  Then $F=\{\pi_{ W_{i}}e_{j}\}_{j\in I_{n},i\in I_{m}}$ is a Parseval frame with a family of optimal dual frames  as $ \{\pi_{V_{i}}e_{j}\}_{j\in I_{n},i\in I_{m}}$.
 Hence, from a given Riesz (orthonormal) fusion basis  we obtain an overcomplete  frame which has several optimal dual frames. See Example \ref{0002}.

\subsection{Robustness of optimal dual fusion frames}
In what follows, we  survey the robustness of optimal dual fusion frames under operator perturbations.  First we recall that, if  $U\in B(\mathcal{H})$  is an invertible operator and $(W,w)$ is  a fusion frame  of $\mathcal{H}$ then  the family $\{(UW, w)\}$ is also a fusion frame  for $\mathcal{H}$, see \cite{Gav07}. Moreover,  in \cite{arefi1}  the stability of dual fusion frames under operator perturbations was considered, although that result needs an extra condition. In the next lemma we present and improve that result.

\begin{lem}\label{lem52}
Let $(W,w)$  be a fusion frame of $\mathcal{H}$  and  $(V,v)$ be a  family of closed subspaces along with a family of weights. Also, let $U\in B(\mathcal{H})$ be an  invertible operator such that $U^{*}U
 W_{i} \subseteq W_{i}$ and $U^{*}U V_{i} \subseteq V_{i}$ for every $i\in I_{m}$. Then  $(V,v)$ is a dual fusion frame of $(W ,w)$ if and only if
 $(UV,v)$ is a dual fusion frame of $(UW ,w)$.
\end{lem}
\begin{proof}
The family $(UW ,w)$ is a fusion frame with the frame operator $US_{W}U^{-1}$,  \cite{Cas08}. Therefore,  applying Lemma \ref{lem3}  we obtain
\begin{eqnarray*}
\sum_{i\in I_{m}} \omega_{i}\upsilon_{i}\pi_{UV_{i}} S_{UW} ^{-1}\pi_{UW_{i}}f =U \sum_{i\in I_{m}}\omega_{i}\upsilon_{i} \pi_{V_{i}} S_{W}^{-1}\pi_{W_{i}}U^{-1}f
\end{eqnarray*}
for each $f\in\mathcal{H}$. This implies the desired result.
\end{proof}

\begin{thm}\label{z}
Let $(W,w)$ be a fusion frame for $\mathcal{H}$ and  $U\in B(\mathcal{H})$ be a unitary  operator. Then $(V,v)$ is  a $1$-loss optimal dual fusion frame of $(W,w)$ if and only if  $\{(UV_{i},\nu_{i})\}_{i\in I_{m}}$ is  a $1$-loss optimal dual  fusion frame of $\{(UW_{i},\omega_{i})\}_{i\in I_{m}}$.\end{thm}
\begin{proof}
Since $U$ is unitary $UW:=\{(UW_{i},\omega_{i})\}_{i\in I_{m}}$ is also a fusion frame  for $\mathcal{H}$, with the frame operator $S_{UW}=US_{W}U^{*}$, see \cite{Gav07}. Moreover,
 if $(V,v)$ is a  dual fusion frame of $(W,w)$ then  $\{(UV_{i},\nu_{i})\}_{i\in I_{m}}$ is a dual fusion frame of
$\{(UW_{i},\omega_{i})\}_{i\in I_{m}}$ by Lemma \ref{lem52}. Moreover, the set of all dual fusion frames of $UW$ is as follows
\begin{eqnarray*}
D_{UW}=\{(UV,v) : (V,v)\in D_{W}\}.
\end{eqnarray*}
To prove the robustness of  optimal dual under this operator, let  $(Z,z)$ be a  dual fusion frame of $(W,w)$
 then
\begin{eqnarray*}
\max_{i\in I_{m}} \Vert  T_{UV}\phi_{uv, uw}M_{i}T^{*}_{UW}\Vert_{F}&=&\max_{i\in I_{m}} \Vert  \omega_{i}\nu_{i}\pi_{UV_{i}}S_{UW}^{-1}\pi_{UW_{i}}\Vert_{F}\\
&=&\max_{i\in I_{m}} \Vert  \omega_{i}\nu_{i}U\pi_{V_{i}}S_{W}^{-1}\pi_{W_{i}}U^{*}\Vert_{F}\\
&=& \max_{i\in I_{m}} \Vert  \omega_{i}\nu_{i}\pi_{V_{i}}S_{W}^{-1}\pi_{W_{i}}\Vert_{F}\\
&\leq&\max_{i\in I_{m}} \Vert  \omega_{i}z_{i}\pi_{Z_{i}}S_{W}^{-1}\pi_{W_{i}}\Vert_{F}\\
&=&  max_{i\in I_{m}} \Vert  \omega_{i}z_{i}U\pi_{Z_{i}}S_{W}^{-1}\pi_{W_{i}}U^{*}\Vert_{F}\\
&=& max_{i\in I_{m}} \Vert  \omega_{i}z_{i}\pi_{UZ_{i}}S_{UW}^{-1}\pi_{UW_{i}}\Vert_{F}\\
&=&\max_{i\in I_{m}} \Vert  T_{UZ}\phi_{uz, uw}M_{i}T^{*}_{UW}\Vert_{F}.
\end{eqnarray*}
Similarly, let $(UV,v)$ be a $1$-loss optimal dual fusion frame of $(UW,w)$. Then for every dual fusion frame $(Z,z)$ of $(W,w)$ we can write
\begin{eqnarray*}
\max_{i\in I_{m}} \Vert T_{V}\phi_{vw}M_{i}T^{*}_{W}\Vert_{F}&=&  \max_{i\in I_{m}} \Vert  T_{UV}\phi_{uv, uw}M_{i}T^{*}_{UW}\Vert_{F}\\
&\leq&\max_{i\in I_{m}} \Vert T_{UZ}\phi_{uz, uw}M_{i}T^{*}_{UW}\Vert_{F}=\max_{i\in I_{m}}  \Vert  T_{Z}\phi_{zw}M_{i}T^{*}_{W}\Vert_{F}.
\end{eqnarray*}
This completes the proof.
\end{proof}
\section{Examples}
In this section we give some examples related to the previous sections. The first example determines the advantage of using optimal dual over $P$-optimal dual fusion frames.
\begin{ex}\label{ex01}(A non-Riesz fusion frame which has several optimal duals but the unique $P$-optimal dual.)

Let   $\{e_{i}\}_{i\in I_{4}}$ be the standard orthonormal basis of $\mathbb{R}^{4}$ and put
\begin{eqnarray*}
W_{1} = {\textit{span}}\{e_{1}, e_{2}\}, \quad W_{2} = {\textit{span}}\{e_{2}, e_{3}\}, \quad W_{3} = {\textit{span}}\{e_{4}\},
\end{eqnarray*}
and  $\omega:=\omega_{i}=1$, $i\in I_{3}$. Then $W = \lbrace (W_{i},\omega_{i})\rbrace_{i\in I_{3}}$ is a fusion frame for $\mathbb{R}^{4}$ and
\begin{equation*}
S_{W}^{-1} = \left[
 \begin{array}{ccc}
1 \quad 0 \quad 0 \quad 0\\

0 \quad \dfrac{1}{2} \quad 0 \quad 0\\

0 \quad 0 \quad 1 \quad 0\\

0 \quad 0 \quad 0 \quad 1\\
\end{array} \right].
\end{equation*}
 Hence
\begin{eqnarray*}
S_{W}^{-1}W_{i} = W_{i}, \quad (i\in I_{4}) .
\end{eqnarray*}
Also,
\begin{equation*}
S_{W}^{-1}\pi_{W_{1}} = \left[
 \begin{array}{ccc}

1 \quad 0 \quad 0 \quad 0\\

0 \quad \dfrac{1}{2} \quad 0 \quad 0\\

0 \quad 0 \quad 0 \quad 0\\

0 \quad 0 \quad 0 \quad 0\\

\end{array} \right],  \quad  S_{W}^{-1} \pi_{W_{2}} = \left[
 \begin{array}{ccc}

0 \quad 0 \quad 0 \quad 0\\

0 \quad \dfrac{1}{2} \quad 0 \quad 0\\

0 \quad 0 \quad 1 \quad 0\\

0 \quad 0 \quad 0 \quad 0\\

\end{array} \right],
\end{equation*}
\begin{equation*}
 S_{W}^{-1} \pi_{W_{3}} = \left[
 \begin{array}{ccc}

0 \quad 0 \quad 0 \quad 0\\

0 \quad 0 \quad 0 \quad 0\\

0 \quad 0 \quad 0 \quad 0\\

0 \quad 0 \quad 0 \quad 1\\

\end{array} \right].
\end{equation*}
 Thus, $\max_{i\in I_{3}} \Vert S_{W}^{-1}\pi_{W_{i}} \Vert_{F}=\sqrt{5/4}$  and  $\Lambda_{1}=\{1,2\}$. Moreover, $\mathcal{H}_{1}\cap \mathcal{H}_{2}=\{0\}$ and $ \lbrace (W_{i}, \omega_{i})\rbrace_{i\in \Lambda_{2}}$ is a Riesz fusion basis for $\mathcal{H}_{2}$. By   Theorem 3.3 of  \cite{morilas} the canonical dual is the unique $P$-optimal dual of $(W,w)$ and consequently it is also an optimal dual, by $(\ref{3.2n})$. However, the canonical dual is not the unique optimal one, and  $(W,w)$ has  many optimal dual fusion frame. For example, take
\begin{eqnarray*}
&&V_{1} = {\textit{span}}\{e_{1}, e_{2}, (0,0,\xi_{1}, \xi_{2})\}, \\
&&  V_{2} = {\textit{span}}\{e_{2}, e_{3},(\xi_{3},0, 0,\xi_{4})\},\\
&& V_{3} = {\textit{span}}\{e_{4},  (\xi_{5},\xi_{6}, \xi_{7},0)\},
\end{eqnarray*}
 for every $\xi_{i}\in \mathbb{R}$, $i\in I_{7}$. Then $V = \lbrace (V_{i},w)\rbrace_{i=1}^{3}$ is a dual fusion frame of $(W,w)$.
Moreover,  for every $\xi_{i}\in \mathbb{R}$, $i\in I_{7}$, the sequence $V$ constitutes a $1$-loss  optimal   dual fusion frame of $(W,w)$. In fact,
\begin{eqnarray*}
\max_{i\in I_{3}} \Vert \pi_{V_{i}}S_{W}^{-1}\pi_{W_{i}}\Vert_{F}=\max_{i\in I_{3}}\Vert S_{W}^{-1}\pi_{W_{2}} \Vert_{F}=\sqrt{5/4}.
\end{eqnarray*}
Note that, in this example $\phi_{vw}$ associated with many dual fusion frames $V$  are not  component preserving. For instance, put $V_{1} = {\textit{span}}\{e_{1}, e_{2}, ce_{4}\}$, $c\neq 0$. Then $\pi_{V_{1}}S_{W}^{-1}\pi_{W_{1}} \subset   V_{1}$, indeed $ce_{4}$ is not in $\pi_{V_{1}}S_{W}^{-1}\pi_{W_{1}}$ and so $\phi_{vw}M_{1}\mathcal{W}\neq M_{1}\mathcal{V}$.
\end{ex}
\begin{ex}\label{0002} (construction an overcomplete frame  with a family of optimal duals  by using an orthonormal fusion basis)

Suppose that $\{e_{j}\}_{j\in I_{3}}$ is the standard orthonormal basis of $\mathbb{R}^{3}$ and
\begin{eqnarray*}
W_{1} = {\textit{span}}\{(1,0,1)\}, \quad W_{2} = {\textit{span}}\{(-1,0,1), (0,1,0)\}.
\end{eqnarray*}
Then $W = \lbrace (W_{i},1)\rbrace_{i=1}^{2}$ is an orthonormal fusion basis for $\mathbb{R}^{3}$. A straightforward computation shows that
\begin{equation*}
{\{\pi_{W_{i}}e_{j}\}_{j\in I_{3},i\in I_{2}}} =\left\{ \left[
 \begin{array}{ccc}

1/2\\
\\
0\\
\\
1/2 \\
\end{array} \right], \left[
 \begin{array}{ccc}

1/2\\
\\
0\\
\\
1/2\\
\end{array} \right], \left[
 \begin{array}{ccc}

1/2\\
\\
0\\
\\
-1/2\\
\end{array} \right], \left[
 \begin{array}{ccc}

0\\
\\
1\\
\\
0\\
\end{array} \right] , \left[
 \begin{array}{ccc}

-1/2\\
\\
0\\
\\
1/2\\
\end{array} \right] \right\}.
\end{equation*}
So, $F:=\{\pi_{W_{i}}e_{j}\}_{j\in I_{3},i\in I_{2}}$ is an overcomplete Parseval  frame and $\{\pi_{V_{i}}e_{j}\}_{j\in I_{3},i\in I_{2}}$ is an optimal dual frame of $F$ for every sequence $\{V_{i}\}_{i=1}^{2}$ so that of $W_{i}\subseteq V_{i}$, by Theorem \ref{orthonormal}. For example $F\in OD_{F}$ and also by putting
\begin{eqnarray*}
V_{1}= {\textit{span}}\{(1,0,1), (1,0,-1)\}\quad, V_{2}=W_{2},
\end{eqnarray*}
 we obtain an optimal  alternate dual frame of $F$ as follows
\begin{equation*}
{G} =\left\{ \left[
 \begin{array}{ccc}

1\\
\\
0\\
\\
0  \\
\end{array} \right], \left[
 \begin{array}{ccc}

0\\
\\
0\\
\\
1\\
\end{array} \right], \left[
 \begin{array}{ccc}

1/2\\
\\
0\\
\\
-1/2\\
\end{array} \right], \left[
 \begin{array}{ccc}

0\\
\\
1\\
\\
0\\
\end{array} \right] , \left[
 \begin{array}{ccc}

-1/2\\
\\
0\\
\\
1/2\\
\end{array} \right] \right\}.
\end{equation*}
\end{ex}
Finally, we present an example of a non-Riesz fusion frame which provides an overcomplete frame  that  its canonical dual is optimal dual for $1$-erasure, however it is not partial optimal for $r$-erasures, for all $r>1$.
\begin{ex}
Let   $\mathcal{H}=\mathbb{R}^{3}$  with the standard orthonormal basis $\{e_{j}\}_{j\in I_{3}}$ and take
\begin{eqnarray*}
W_{1} ={\textit{span}}\{e_{1}, e_{2}\}, \quad W_{2} = {\textit{span}}\{e_{2}\}, \quad W_{3} = {\textit{span}}\{e_{3}, e_{1}- e_{2}\},
\end{eqnarray*}
and  $\omega_{i}=1$, $i\in I_{3}$. Then $W = \lbrace (W_{i},\omega_{i})\rbrace_{i\in I_{3}}$ is a fusion frame for $\mathcal{H}$ and
\begin{equation*}
S_{W}^{-1} = \left[
 \begin{array}{ccc}
\dfrac{5}{7} \quad \dfrac{1}{7} \quad 0\\

\\
\dfrac{1}{7} \quad \dfrac{3}{7} \quad 0\\

\\
0 \quad 0 \quad 1 \\

\end{array} \right].
\end{equation*}
Hence, we compute $F=\{\pi_{W_{i}}S_{W}^{-1}e_{j}\}_{i,j}$ as follows
\begin{equation*}
{F} =\left\{ \left[
 \begin{array}{ccc}

5/7\\
\\
1/7 \\
\\
0  \\
\end{array} \right], \left[
 \begin{array}{ccc}

1/7\\
\\
3/7\\
\\
0\\
\end{array} \right], \left[
 \begin{array}{ccc}

0\\
\\
1/7\\
\\
0\\
\end{array} \right], \left[
 \begin{array}{ccc}

0\\
\\
3/7\\
\\
0\\
\end{array} \right] , \left[
 \begin{array}{ccc}

2/7\\
\\
-2/7\\
\\
0\\
\end{array} \right], \left[
 \begin{array}{ccc}

-1/7\\
\\
1/7\\
\\
0\\
\end{array} \right], \left[
 \begin{array}{ccc}

0\\
\\
0\\
\\
1\\
\end{array} \right] \right\}
\end{equation*}
and
\begin{equation*}
{S_{F}^{-1}F} =\left\{ \left[
 \begin{array}{ccc}

1\\
\\
0\\
\\
0  \\
\end{array} \right], \left[
 \begin{array}{ccc}

0\\
\\
1\\
\\
0\\
\end{array} \right], \left[
 \begin{array}{ccc}

1/10\\
\\
3/10\\
\\
0\\
\end{array} \right], \left[
 \begin{array}{ccc}

3/10\\
\\
9/10\\
\\
0\\
\end{array} \right] , \left[
 \begin{array}{ccc}

4/5\\
\\
-2/5\\
\\
0\\
\end{array} \right], \left[
 \begin{array}{ccc}

-2/5\\
\\
1/5\\
\\
0\\
\end{array} \right], \left[
 \begin{array}{ccc}

0\\
\\
0\\
\\
1\\
\end{array} \right] \right\}.
\end{equation*}
Then we obtain $\max_{i\in I_{7}}\Vert S_{F}^{-1}f_{i}\Vert\Vert f_{i}\Vert=\Vert S_{F}^{-1}f_{7}\Vert\Vert f_{7}\Vert=1$. By taking a sequence $\{u_{i}\}_{i\in I_{7}}$ satisfies $(\ref{dual relation})$ one derives the following relations
\begin{eqnarray}\label{ufi}
&&5u_{1}+u_{2}+2u_{5}-u_{6}=0,\\
&& u_{1}+3u_{2}+u_{3}+3u_{4}-2u_{5}+u_{6}=0,\\
&&u_{7}=0.
\end{eqnarray}
Due to $(5.3)$ for every dual frame $\{g_{i}\}_{i\in I_{7}}=\{S_{F}^{-1}f_{i}+u_{i}\}_{i\in I_{7}}$ of $F$ we have that
\begin{equation*}
\max_{i\in I_{7}}\Vert S_{F}^{-1}f_{i}\Vert\Vert f_{i}\Vert=\Vert g_{7}\Vert\Vert f_{7}\Vert \leq
\max_{i\in I_{7}}\Vert g_{i}\Vert\Vert f_{i}\Vert,
\end{equation*}
i.e., $S_{F}^{-1}F\in OD_{F}$. We show that the canonical dual cannot be considered as a partial optimal dual for $2$-erasures. To this end,  let $i$th and $j$th component are lost. we set $u_{i}=\dfrac{-1}{2}S_{F}^{-1}f_{i}$ and $u_{j}=\dfrac{-1}{2}S_{F}^{-1}f_{j}$ then one obtain $u_{k}$, $k\in I_{7}$, $k \neq  i,j$ by $(5.1)$,  $(5.2)$ and $(5.3)$. So $\{g_{i}\}_{i\in I_{7}}=\{S_{F}^{-1}f_{i}+u_{i}\}_{i\in I_{7}}$ is a dual frame of $F$. Suppose $D$ is a $7 \times7$ diagonal matrix with  $d_{ii}=d_{jj}=1$  and other matrix elements $0$. Then
\begin{eqnarray*}
\left\Vert T_{S_{F}^{-1}F}DT^{*}_{F}\right\Vert_{F}&=& \left\Vert \sum_{k=i,j}S_{F}^{-1}f_{k}\otimes f_{k}\right\Vert_{F}\\
&>&\dfrac{1}{2}\left\Vert \sum_{k=i,j}S_{F}^{-1}f_{k}\otimes f_{k}\right\Vert_{F}\\
&=& \left\Vert \sum_{k=i,j}g_{k}\otimes f_{k}\right\Vert_{F}\\
&=&\left\Vert T_{G}DT^{*}_{F}\right\Vert_{F}.
\end{eqnarray*}
Thus the canonical dual is not a partial optimal dual for $2$-erasures and so for any $r$-erasures, $r>1$.
\end{ex}

\bibliographystyle{amsplain}

\end{document}